\pdfoutput=1
\RequirePackage{ifpdf}
\ifpdf 
\documentclass[pdftex]{sigma}
\else
\documentclass{sigma}
\fi

\usepackage[all]{xy}
\usepackage{extarrows}
\usepackage{upgreek}

\numberwithin{equation}{section}

\newtheorem{Theorem}{Theorem}[section]

\newtheorem{Proposition}[Theorem]{Proposition}
 { \theoremstyle{definition}
\newtheorem{Definition}[Theorem]{Definition}

\newtheorem{Example}[Theorem]{Example}
\newtheorem{Remark}[Theorem]{Remark} }

\begin{document}

\allowdisplaybreaks

\renewcommand{\thefootnote}{$\star$}

\newcommand{\arXivNumber}{1412.5981}

\renewcommand{\PaperNumber}{079}

\FirstPageHeading

\ShortArticleName{Lie Algebroids in the Loday--Pirashvili Category}

\ArticleName{Lie Algebroids in the Loday--Pirashvili Category\footnote{This paper is a~contribution to the Special Issue
on Poisson Geometry in Mathematics and Physics.
The full collection is available at \href{http://www.emis.de/journals/SIGMA/Poisson2014.html}{http://www.emis.de/journals/SIGMA/Poisson2014.html}}}

\Author{Ana ROVI~$^{\dag\ddag}$}

\AuthorNameForHeading{A.~Rovi}

\Address{$^\dag$~School of Mathematics and Statistics, University of Glasgow,\\
\hphantom{$^\dag$}~University Gardens, Glasgow G12 8QW, UK}

\Address{$^\ddag$~School of Mathematics and Statistics, Newcastle University,\\
\hphantom{$^\ddag$}~Newcastle upon Tyne NE1 7RU, UK}
\EmailD{\href{mailto:ana.rovi-garcia@ncl.ac.uk}{ana.rovi-garcia@ncl.ac.uk}}
\URLaddressD{\url{https://www.staff.ncl.ac.uk/ana.rovi-garcia/}}

\ArticleDates{Received February 27, 2015, in f\/inal form September 28, 2015; Published online October 02, 2015}

\Abstract{We describe Lie--Rinehart algebras in  the tensor category $\mathcal{LM}$ of linear maps in the sense of Loday and Pirashvili and
construct a functor from  Lie--Rinehart algebras  in~$\mathcal{LM}$ to Leibniz algebroids.}


\Classification{17A32; 53D17}

\renewcommand{\thefootnote}{\arabic{footnote}}
\setcounter{footnote}{0}

\section{Introduction}

Leibniz algebras were introduced by Loday and Pirashvili \cite{LodayLeibnizSurvey, LodayLeibnizEnveloping} as a non skew-symmetric generalisation of Lie algebras.  Leibniz algebras appear in dif\/ferential geometry as the algebraic structure on Courant algebroids~\cite{WeinsteinCourant}   and in mathematical physics, for example in Chern--Simons theory~\cite{Jose}.

Loday and Pirashvili \cite{LodayLeibnizEnveloping} observed  that Leibniz algebras  can be described   as the canonical map $\pi\colon \mathfrak g \rightarrow \mathfrak g_{\mathrm{Lie}}$ where $\mathfrak g$ is a Leibniz algebra and $\mathfrak g_{\mathrm{Lie}}$ is the Lie algebra which arises as  the quotient of $\mathfrak g$ by the Leibniz ideal generated by elements $[ x , x ]_{\mathfrak g}$ for $x \in \mathfrak g$. This observation leads to their def\/inition~\cite{Loday98} of the monoidal category $\mathcal{LM}$ of linear maps and to the construction of a~pair of adjoint functors between Lie algebras and Leibniz algebras. Note that the category~$\mathcal{LM}$ can be seen as the category of truncated chain complexes of length one.

In this paper, we focus on the interplay between  Lie--Rinehart algebras \cite{HuebschmannTerm} (Lie algebroids \cite{Pradines} in a dif\/ferential geometric context) and Leibniz algebroids \cite{BaragliaLeibniz, LeonLeibniz} which are  generalisations of  Lie algebras and Leibniz algebras respectively. Our goal is not to def\/ine a dif\/ferential geometric counterpart of Loday and Pirashvili's  category of linear maps but to use their construction to   describe (and understand) some interesting relations between Lie--Rinehart algebras and Leibniz algebras. In this sense, we describe Lie--Rinehart algebras in~$\mathcal{LM}$ and then   construct a   functor from  Lie--Rinehart algebras in~$\mathcal{LM}$ to  Leibniz algebroids.

Throughout, let $R$ be a unital commutative ring and let an unadorned $\otimes$ denote $\otimes_R$.

\begin{Theorem}
\label{thm1}
Let $( M \overset{g}{\rightarrow} A)$ be  a commutative $R$-algebra object and $( N \overset{ f}{\rightarrow} L)$ be a Lie algebra object in $\mathcal{LM}$.  A pair  $\big( ( M \overset{g}{\rightarrow} A) , ( N \overset{ f}{\rightarrow} L) \big)$ is a Lie--Rinehart algebra object in $\mathcal{LM}$ if

\begin{itemize}\itemsep=0pt
\item $( A , L)$ is a Lie--Rinehart algebra  with anchor $\rho_0 \colon L \rightarrow \mathrm{Der}_R(A)$,

\item the $A$-bimodule $M$ is a left $(A,L)$-module with action given by $\rho_2 \colon L  \rightarrow  \mathrm{Hom}_R (M,M) $,

\item the right $L$-module $N$ with action $N \otimes L \rightarrow N$ given by $n \otimes \xi \mapsto [n , \xi ]$ for all $n \in N$, $\xi \in L$  is also a left $A$-module with  $[ - , - ]$   satisfying
\begin{gather*}
 [ a \cdot n ,  \xi ]  = a \cdot [ n , \xi ] - \rho_0 ( \xi ) ( a ) \cdot n, \qquad \forall\, a \in A,
\end{gather*}

\item both  $f$ and $g$ are $L$-equivariant and $A$-linear,
\end{itemize}
and there exist
\begin{itemize}\itemsep=0pt
\item an $A$-module map $\lambda \colon M \otimes_A L \rightarrow N$,

\item an $A$-module map  $\rho_1 \colon N \rightarrow \mathrm{Der}_R (A ,M)$ satisfying
\begin{gather*}
g ( \rho_1 (n)(a) ) = \rho_0 ( f(n) ) (a) , \qquad   \rho_1 \left( [ n , \xi ] \right) = [ \rho_1 (n) , ( \rho_0 + \rho_2 ) ( \xi ) ]
\end{gather*}
for all $a \in A$, $m \in M$, $n \in N$, $\xi \in L$.
\end{itemize}
\end{Theorem}

Our second result consists in the construction of a functor from the category of Lie--Rinehart algebra objects in~$\mathcal{LM}$ to the category of Leibniz algebroids.

\begin{Theorem}
\label{thm2}
For a Lie--Rinehart algebra $\big( (M \overset{ g}{\rightarrow} A), ( N \overset{ f}{\rightarrow} L) \big)$ in  $\mathcal{LM}$, the pair $( A, M \oplus N) $ is a Leibniz algebroid with anchor
\begin{gather}
\label{equiv rhoN}
\rho_{M \oplus N} := -   \rho_0 \circ  f \colon \  M \oplus N  \longrightarrow \mathrm{Der}_R(A)
\end{gather}
and Leibniz bracket on the $A$-module $M \oplus N$ given by
\begin{gather}
\label{eqNbracket}
[m_1 + n_1 , m_2 + n_2]_{M \oplus N} := -  \rho_2(f( n_2)) (m_1) +  [ n_1 , f(n_2)].
\end{gather}
for all $m_1 , m_2 \in M$ and $n_1, n_2 \in N$.
\end{Theorem}

Since a very rich class of examples of Hopf algebroids \cite{GabiSzlachanyi, kr1, R2} is the enveloping algebra of Lie--Rinehart algebras, following the argument given by Loday and Pirashvili in~\cite{Loday98}, we would expect a similar relation  between the enveloping algebra of a Leibniz algebroid  and Hopf algebroids in~$\mathcal{LM}$. However this generalisation of~\cite{Loday98} requires not only the def\/inition of a~correct notion of enveloping algebra of a Leibniz algebroid (not yet def\/ined in the literature), but also the construction of a functor from the category of Leibniz algebroids to the category of Lie--Rinehart algebras objects in~$\mathcal{LM}$, which goes beyond the goals of this paper.

\section{Leibniz algebroids}\label{VL}

In this section we f\/irst recall the def\/initions of Leibniz algebras  as given by Loday and Pirashvili~\cite{LodayLeibnizSurvey, LodayLeibnizEnveloping}. Secondly, we discuss Leibniz algebroids, see~\cite{LeonLeibniz} for a dif\/ferential geometric description, and  give some motivating examples.

\subsection{Leibniz algebras}
\label{Leibniz history}

 Leibniz algebras  were f\/irst def\/ined by Blokh \cite{Blokh}, later rediscovered  and more intensively  studied  since \cite{Cuvier,LodayLeibnizEnveloping}.  For motivation, def\/initions and basic examples see \cite{LodayLeibnizSurvey, LodayLeibnizEnveloping}.

\begin{Definition}
A right Leibniz algebra $\mathfrak g $  is an $R$-module equipped with a bilinear map, called the right Leibniz bracket and denoted  by  $ [ - , - ]_{\mathfrak g} \colon  \mathfrak g \otimes \mathfrak g \longrightarrow \mathfrak g $ which
satisf\/ies the  identity
\begin{gather}
\label{RLJ}
[ x , [ y , z ]_{\mathfrak g} ]_{\mathfrak g} - [ [ x , y ]_{\mathfrak g} , z ]_{\mathfrak g} + [ [ x , z ]_{\mathfrak g} , y ]_{\mathfrak g} = 0, \qquad \textnormal{for all $x , y , z \in \mathfrak g$}.
\end{gather}
\end{Definition}

Correspondingly, a \emph{left} Leibniz algebra structure $[-,-]$ on an $R$-module $V$  is def\/ined by  $[x , y ] := [ y , x ]_{\mathfrak g}$ for $x , y \in V$ where $[ - , - ]_{\mathfrak g}$ satisf\/ies \eqref{RLJ}, see~\cite{LodayLeibnizSurvey, LodayLeibnizEnveloping}.

Since Loday and  Pirashvili  use right  Leibniz algebra structures in their work \cite{LodayLeibnizSurvey,LodayLeibnizEnveloping, Loday98}, we  choose right Leibniz algebras as well, which we  call from now on Leibniz algebras.

\begin{Remark}
For a  Leibniz algebra $\mathfrak g$  there exists a corresponding Lie algebra, denoted by $\mathfrak g_{\mathrm{Lie}}$ and called the reduced Lie algebra of $\mathfrak g$, which arises by taking the quotient of $\mathfrak g$ by the Leibniz  ideal generated by elements $[ x , x ]_{\mathfrak g} \in \mathfrak g $ for $x \in \mathfrak g$.
Hence there exists a surjective map
\begin{gather*}
 \pi\colon \ \mathfrak g \longrightarrow \mathfrak g_{\mathrm{Lie}}.
\end{gather*}
\end{Remark}

\begin{Example}[see \cite{KurdianiPirashvili}]
Let $L$  be a Lie algebra over $R$ with bracket $ [ - , - ]_L $. The bracket on the second tensor power of $L$   given by
\begin{gather}
\label{bracket LL}
[ x_1 \otimes y_1 , x_2 \otimes y_2 ]_{L \otimes L} = [ x_1 , [ x_2 , y_2 ]_L ]_L \otimes y_1 + x_1 \otimes [ y_1 , [ x_2 , y_2 ]_L ]_L
\end{gather}
for all $x_1 , x_2 , y_1 , y_2 \in L$  endows $L \otimes  L$  with a Leibniz algebra structure.
\end{Example}

In the following example we reformulate the construction of the  \emph{hemi-semi-direct product} for left Leibniz algebras  introduced by Kinyon and Weinstein in  \cite[Example~2.2]{KinyonWeinstein} and endow  the direct sum of a Lie algebra $L$ and a (left) $L$-module $V$ with a  (right) Leibniz algebra structure.

\begin{Example}
\label{hemisemi}
Let $ L$  be a Lie algebra over $R$  and   $V$ be a  $L$-module with left action $ L \otimes V \rightarrow V$ given by $\xi \otimes a \mapsto \xi ( a ) $ for all $a \in V$ and $\xi \in L$.   The direct sum (of $R$-modules) $V \oplus L$ together with the bracket
\begin{gather*}
\left[ a +  \xi , b + \zeta \right]_{V \oplus L} :=  \zeta (a) - [\xi , \zeta]_L, \qquad a, b \in V, \xi , \zeta \in L
\end{gather*}
becomes a (right)  Leibniz algebra   since the identity  \eqref{RLJ} is satisf\/ied
\begin{gather*}
  [ a + \xi , [ b  + \zeta , c + \gamma]_{V \oplus L} ]_{V \oplus L}\! - [ [ a + \xi , b + \zeta ]_{V \oplus L} , c + \gamma ]_{V \oplus L } + [ [ a + \xi , c + \gamma ]_{V \oplus L} , b + \zeta ]_{V \oplus L}  \\
 \qquad{} = [ a + \xi ,  \gamma ( b ) - [ \zeta, \gamma ]_L ]_{V \oplus L}  - [  \zeta ( a ) - [ \xi , \zeta ]_L , c + \gamma  ]_{V \oplus L } + [ \gamma(a) - [ \xi , \gamma ]_L ,  b + \zeta ]_{V \oplus L} \\
\qquad{} = - [ \zeta , \gamma]_L (a) + [ \xi , [ \zeta , \gamma ]_L ]_L  - \gamma ( \zeta (a) ) - [ [ \xi , \zeta ]_L , \gamma ]_L + \zeta ( \gamma(a) ) + [ [ \xi , \gamma ]_L , \zeta ]_L = 0. \tag*{\qed}
  \end{gather*}
  \renewcommand{\qed}{}
\end{Example}

\begin{Definition}[see \cite{LodayLeibnizSurvey}]
Let $\mathfrak g$ and $\mathfrak g'$ be Leibniz algebras. A map of  Leibniz algebras $\varphi\colon  \mathfrak g \rightarrow \mathfrak g'$  is a homomorphism of $R$-modules satisfying  $ \varphi ([ x , y ]_{\mathfrak g} ) = [ \varphi(x) , \varphi (y) ]_{\mathfrak g'} $ for all $x , y \in \mathfrak g$.
\end{Definition}

\begin{Proposition}
\label{hemi2}
Let $\mathfrak g$ be a Leibniz algebra over $R$ and let $M$ be a left module over its reduced Lie algebra $\mathfrak g_{\mathrm{Lie}}$ with left action  $ \mathfrak g_{\mathrm{Lie}}  \otimes M \rightarrow M$ given by $\pi(g ) \otimes m \mapsto \pi (g) ( m ) $ for all $m \in M$ and $g  \in \mathfrak g$.   The direct sum $($of $R$-modules$)$ $M \oplus \mathfrak g $ together with the bracket
\begin{gather}
\label{hseq2}
\left[ m_1 +  g_1  , m_2 + g_2 \right]_{M \oplus \mathfrak g} := - \pi(g_2) (m_1) + [ g_1 , g_2]_\mathfrak g, \qquad m_1, m_2 \in M, g_1, g_2  \in \mathfrak g
\end{gather}
is a Leibniz algebra.
\end{Proposition}

\begin{proof}
Since $\pi\colon  \mathfrak g \rightarrow \mathfrak g_{\mathrm{Lie}}$ is a map of Leibniz algebras, a straightforward computation identical to the one carried out in Example \ref{hemisemi} yields that $[ - , - ]_{M \oplus \mathfrak g}$ satisf\/ies  \eqref{RLJ} and is hence a Leibniz bracket.
\end{proof}

\subsection{Leibniz algebroids and related structures}

While Lie algebras can be generalised to Lie--Rinehart algebras \cite{Herz, HuebschmannTerm,RinehartForms} (Lie algebroids \cite{Pradines} in dif\/ferential geometric context), Leibniz algebras \cite{Blokh,LodayLeibnizEnveloping} give rise to dif\/ferent algebraic objects:   Leibniz algebroids, f\/irst def\/ined in a dif\/ferential geometric context in~\cite{LeonLeibniz};  Loday algebroids~\cite{StienonLoday};
 Courant algebroids~\cite{WeinsteinCourant}; Courant--Dorfman algebras, a term coined by Roytenberg~\cite{Roytenberg} to denote a structure encompassing both Courant algebroids~\cite{WeinsteinCourant} and Dorfman algebras.  See~\cite{YKSCourant} for a~description of the  historic development of these structures.

\subsubsection{Leibniz algebroids}

We propose a def\/inition of Leibniz algebroids in purely algebraic terms, following the def\/initions given by Rinehart \cite{RinehartForms} and later by Huebschmann \cite{HuebschmannTerm} for Lie--Rinehart algebras as an algebraic description of Lie algebroids.

\begin{Definition}
\label{LBdef}
Let $A$ be a commutative $R$-algebra and $\mathcal E$ be a Leibniz algebra over $R$ with bracket $ [ - , - ]_{\mathcal E}$. The pair $(A , \mathcal E)$ is  called a \textit{Leibniz algebroid} if the Leibniz algebra  $\mathcal E$ has a left $A$-module structure $\mu\colon  A \otimes \mathcal E \rightarrow \mathcal E$ given by $ a \otimes e \mapsto a \cdot e$ for all $a \in A$ and $e \in \mathcal E$, and there exists an $A$-linear Leibniz algebra antihomomorphism $\rho_{\mathcal E} \colon  \mathcal E \rightarrow \mathrm{Der}_R(A)$, called the anchor, satisfying
\begin{gather}
\label{LBrule}
[ a \cdot e_1 , e_2 ]_{\mathcal E} = a \cdot [ e_1 , e_2]_{\mathcal E} + \rho_{\mathcal E}(e_2) (a) \cdot e_1, \qquad \textnormal{for} \ \ e_1 , e_2 \in \mathcal E, \quad a \in A.
\end{gather}
\end{Definition}

\begin{Example}
A Lie--Rinehart algebra  $(A,L)$, with anchor $\rho_L$, is a Leibniz algebroid with anchor $ - \rho_L$.
\end{Example}

\begin{Proposition}
\label{AAL algebroid}
Let $(A,L)$ be a Lie--Rinehart algebra, with anchor $\rho_L$, and let $M$ be a left $(A,L)$-module with action $L \otimes M \rightarrow M$ given by $\xi \otimes m \mapsto \nabla^\ell_\xi (m)$.  The pair $(A , M \oplus L)$  is a~Leibniz algebroid with anchor
\begin{gather}
\label{LBanchor}
\rho_{M \oplus L} ( m + \xi) := - \rho_L ( \xi)
\end{gather}
 and bracket on the direct sum $M \oplus L$   given by the $($negative$)$ hemi-semi-direct product  $[ - , - ]_{M \oplus L}$ of $M$ by $L$ with action $L \otimes M \rightarrow M$ $($or equivalently $L \rightarrow \mathrm{Hom}_R ( M , M))$ given by $\nabla^\ell$.
\end{Proposition}

\begin{proof}
First note that $[ m_1 + \xi , m_2 ]_{M \oplus L} = 0$ for all $m_1 , m_2  \in M$, $\xi \in L$. Now, since $M$ is an $(A,L)$-module with action $\nabla^\ell \colon  L \otimes M \rightarrow M$ given by $m \otimes \xi \mapsto \nabla^\ell_\xi (m)$, we endow the direct sum $M \oplus L$ with  the  Leibniz bracket given in  \eqref{hseq2}, that is
\begin{gather*}
[ m_1 + \xi , m_2 + \zeta ]_{M \oplus L} = - \nabla^\ell_\zeta (m_1) + [ \xi , \zeta ]_L.
\end{gather*}
We now check that the map in \eqref{LBanchor} is an antihomomorphism
\begin{gather*}
\rho_{M \oplus L} ( [ m_1 + \xi , m_2 + \zeta]_{M \oplus L} )    = \rho_{M \oplus L} ( - \nabla^\ell_\zeta (m_1) +   [ \xi , \zeta ]_L)  = - \rho_L ( [ \xi , \zeta]_L )  \\
\hphantom{\rho_{M \oplus L} ( [ m_1 + \xi , m_2 + \zeta]_{M \oplus L} )}{}
   = - [ \rho_L ( \xi ) , \rho_L ( \zeta) ]_{\mathrm{Der}_R(A)} \\
\hphantom{\rho_{M \oplus L} ( [ m_1 + \xi , m_2 + \zeta]_{M \oplus L} )}{}
 =  [ \rho_L ( \zeta ) , \rho_L ( \xi )  ]_{\mathrm{Der}_R(A)}  =  [ \rho_{ M \oplus L} ( \zeta ) , \rho_{ M \oplus L} ( \xi )  ]_{\mathrm{Der}_R(A)} .
\end{gather*}
Lastly, we check that the compatibility condition beween $[ - , - ]_{M \oplus L}$ and the $A$-module structure on $M \oplus L$ given  in \eqref{LBrule} is satisf\/ied
\begin{gather*}
[a \cdot ( m + \xi ) ,  \zeta ]_{M \oplus L}    = - \nabla^\ell_\zeta ( a \cdot m) +  [ a \cdot \xi , \zeta ]_L  \\
\hphantom{[a \cdot ( m + \xi ) ,  \zeta ]_{M \oplus L}    }{}
 =  - a \cdot \nabla^\ell_\zeta(m) - \rho_L ( \zeta)(a) \cdot m +  a \cdot [ \xi , \zeta ]_L - \rho_L ( \zeta)  ( a ) \cdot \xi \\
\hphantom{[a \cdot ( m + \xi ) ,  \zeta ]_{M \oplus L}   }{}
 =  a \cdot ( - \nabla^\ell_\zeta(m) + [ \xi , \zeta ]_L ) - \rho_L (  \zeta ) ( a ) \cdot ( m + \xi ) \\
\hphantom{[a \cdot ( m + \xi ) ,  \zeta ]_{M \oplus L}  }{}
 = a \cdot [m + \xi , \zeta ]_{M \oplus L} + \rho_{M \oplus L} ( \zeta) ( a ) \cdot ( m + \xi ).\tag*{\qed}
  \end{gather*}
  \renewcommand{\qed}{}
\end{proof}

Note that the Leibniz rule for $[ - , - ]_{M \oplus L}$ given by
\begin{gather*}
[m + \xi , a \cdot \zeta ]_{M \oplus L}   = - \nabla^\ell_{a \cdot \zeta } (m) + [ \xi , a \cdot \zeta ]_L
  = - a \cdot \nabla_\zeta^\ell (m) + a \cdot [ \xi , \zeta ]_L  + \rho_L ( \xi ) ( a ) \cdot \zeta \\
\hphantom{[m + \xi , a \cdot \zeta ]_{M \oplus L} }{}
  = a \cdot [ m + \xi , \zeta ]_{M \oplus L } - \rho_{M \oplus L } ( \xi ) ( a ) \cdot \zeta
\end{gather*}
 for all $a \in A$ and $\xi , \zeta \in L$  implies that the Leibniz algebroid $(A , M \oplus L)$ is \emph{local} in the sense of \cite[Def\/inition~3.4]{BaragliaLeibniz}.

In general, the relations between Lie algebras and Leibniz algebras will not induce relations between corresponding Lie--Rinehart algebras and Leibniz algebroids.

\begin{Example}
Let $(A , L)$ be a Lie--Rinehart algebra, with anchor $\rho_L$. The pair $(A , L \otimes L)$ where $L \otimes L$ is the Leibniz algebra with bracket  given by~\eqref{bracket LL}
 will not be  a Leibniz algebroid in general. Since $[ - , - ]_L$ satisf\/ies the Leibniz rule~\eqref{LBrule}, we have
\begin{gather*}
[ a \cdot x_1 \otimes y_1 , x_2 \otimes y_2 ]_{ L \otimes L}   =
 [ a \cdot x_1 , [ x_2 , y_2 ]_L ]_L \otimes y_1 + a \cdot x_1 \otimes [ y_1 , [ x_2 , y_2 ]_L ]_L  \\
\qquad{} = a \cdot [ x_1 , [ x_2 , y_2 ]_L ]_L \otimes y_1   - \rho_L \left( [ x_2 , y_2 ]_L \right) ( a ) \cdot x_1 \otimes y_1
+ a \cdot x_1 \otimes [ y_1 , [ x_2 , y_2 ]_L ]_L \\
\qquad{} =  a \cdot \left(  [ x_1 , [ x_2 , y_2 ]_L ]_L \otimes y_1 +  x_1 \otimes [ y_1 , [ x_2 , y_2 ]_L ]_L \right)
  - \rho_L \left( [ x_2 , y_2 ]_L \right) ( a ) \cdot x_1 \otimes y_1  \\
\qquad{} = a \cdot [ x_1 \otimes y_1 , x_2 \otimes y_2 ]_{L \otimes L} - \rho_L ( [ x_2 , y_2 ] ) (a) \cdot x_1 \otimes y_1,
\end{gather*}
 but the map $\gamma ( x_2 \otimes y_2 ) := \rho_L ( [ x_2 , y_2 ]_L)$ is not $A$-linear since $[ - , - ]_L$ is not, hence the pair $(A , L \otimes L)$ does not admit an anchor map induced by~$\rho_L$.
\end{Example}

From \cite{Loday98} we know that to each Leibniz algebra $\mathfrak g$ we can canonically associate a Lie algebra~$\mathfrak g_{\mathrm{Lie}}$ by taking the quotient of $\mathfrak g$ by the two-sided ideal $[ x , x ] $ for $x \in \mathfrak g$. This  relation does not generalise to a canonical relation between  Leibniz algebroids and Lie--Rinehart algebras, so that given a Leibniz algebroid $(A , \mathcal E)$, the reduced Lie algebra $\mathcal E_{\mathrm{Lie}} $ will not be compatible in general with $A$.

\begin{Proposition}
Let $(A, \mathcal E)$ be a Leibniz algebroid  with anchor $\rho_{\mathcal E}$. If $\mathrm{Ker}( \pi)$ is an $A$-sub\-module of $\mathcal E$, then the pair $(A, \mathcal E_{\mathrm{Lie}})$ is a Lie--Rinehart algebra with anchor denoted by $\rho_{\mathcal E_{\mathrm{Lie}}}$ and given by $- \rho_{\mathcal E}$.
 \end{Proposition}

\begin{proof}
First note that the anchor $\rho_{\mathcal E}$ descends to an $R$-linear map $\gamma_{\mathcal E_{\mathrm{Lie}}}\colon  \mathcal E_{\mathrm{Lie}} \rightarrow \mathrm{Der}_R(A)$ since   $ \rho_{\mathcal E} \left( [ e , e ]_{\mathcal E} \right) = [ \rho_{\mathcal E} (e) , \rho_{\mathcal E} (e)]_{\mathcal E} = 0 $ for all $ e \in \mathcal E$.   Let us  assume that $\pi\colon  \mathcal E \rightarrow \mathcal E_{\mathrm{Lie}}$ is $A$-linear. Then we have
\begin{gather*}
\gamma_{\mathcal E_{\mathrm{Lie}}} ( a \cdot \pi (e) ) = \gamma_{\mathcal E_{\mathrm{Lie}}} \left( \pi ( a \cdot e ) \right) = \rho_{\mathcal E} ( a \cdot e ) = a \cdot \rho_{\mathcal E} (e) = a \cdot \gamma_{\mathcal E_{\mathrm{Lie}}} ( \pi (e) ),
\end{gather*}
so that $\gamma_{\mathcal E_{\mathrm{Lie}}}\colon  \mathcal E_{\mathrm{Lie}} \rightarrow \mathrm{Der}_R (A)$ is $A$-linear. Since $\rho_{\mathcal E}$ is a antihomomorphism while the anchor of a Lie--Rinehart algebra is a homomorphism, we set $\rho_{\mathcal E_{\mathrm{Lie}}} := - \gamma_{\mathcal E_{\mathrm{Lie}}}$ so that  $(A, \mathcal E_{\mathrm{Lie}})$ is a~Lie--Rinehart algebra and the following diagram
\begin{gather*}
  \xymatrix{
 \mathcal E \ar[r]^-{\rho_{\mathcal E}} \ar[d]_\pi & \mathrm{Der}{A}  \\
    \mathcal E_{\mathrm{Lie}}  \ar[ur]_-{- \rho_{\mathcal E_{\mathrm{Lie}}}}  &   }
\end{gather*}
commutes.
\end{proof}

Note that   given a Leibniz algebroid $(A, \mathcal E)$, the $A$-module structure on $\mathcal E$ will  not descend to an $A$-module structure on the reduced Lie algebra $\mathcal E_{\mathrm{Lie}}$ in general.

\begin{Example}
\label{TM T*M}
On a manifold $M$, the bundle $E= TM \oplus T^*M$ has a natural Courant--Dorfman algebra structure with:
\begin{itemize}\itemsep=0pt
\item a bilinear form given by $\langle X + \xi , Y + \zeta \rangle = \iota_X \zeta + \iota_Y \zeta $, where $X , Y \in TM$, $\xi , \zeta \in T^*M$ and $\iota_X \xi$ is the contraction of $X$ with $\xi$,

\item a derivation $d \colon  C^\infty (M) \rightarrow T^*X$ given by the dif\/ferential of a function,

\item a right Leibniz bracket given by
$
[ X + \xi , Y + \zeta ] = [X, Y ] - \mathcal L_X \zeta + \mathcal L_Y \xi + d \left( \iota_X \zeta \right)$ where $[ - , - ]$ is the commutator of vector f\/ields and $\mathcal L$ is the Lie derivative.
\end{itemize}
Note that we have  $[ X + \xi , X + \xi ] = d ( i_X \xi ) $ so that the reduced Lie algebra $E_{\mathrm{Lie}}$ corresponding to $ E= TM \oplus T^*M$ is
\begin{gather*}
E_{\mathrm{Lie}}= TM \oplus T^*M / \langle \textrm{exact forms} \rangle
\end{gather*}
with Lie bracket given by $[ X + \xi , Y + \zeta ] = [X, Y ] - \mathcal L_X \zeta + \mathcal L_Y \xi $. In particular, we see that the  map $\pi\colon  E \rightarrow E_{\mathrm{Lie}}$ is not $C^\infty (M)$-linear.   Note also that while $[ X + \xi , df ] = 0$, we have
\begin{gather*}
[ X + \xi , a \cdot df ] = - \langle X , da \rangle \cdot df + \langle X , df \rangle \cdot da.
\end{gather*}
Note   that the reduced Lie algebra $ E_{\mathrm{Lie}}$ is not an $C^\infty(M)$-module so that $( C^\infty (M) , E_{\mathrm{Lie}} )$ is not a Lie--Rinehart algebra.
\end{Example}

\section[Lie-Rinehart algebras  in the category $\mathcal{LM}$ of linear maps]{Lie--Rinehart algebras  in the category $\boldsymbol{\mathcal{LM}}$ of linear maps}
\label{LRLM}

Lie--Rinehart algebras \cite{Herz,HuebschmannTerm, RinehartForms} (Lie algebroids~\cite{Pradines} in dif\/ferential geometric context) were  introduced by Herz \cite{Herz} under the name \emph{Lie pseudo--algebra} (also known as~\emph{Lie \index{Lie algebroids} algebroid} \cite{Pradines} in a dif\/ferential geometric context) and has been developed and studied as a  generalisation of Lie algebras. The term \emph{Lie--Rinehart algebra} was coined by Huebschmann \cite{HuebschmannTerm}, a term which acknowledges Rinehart's fundamental contributions~\cite{RinehartForms} to the understanding of this structure. See \cite[Section~1]{HuebschmannPCQ} for some historical remarks on this development.

We start this section by giving an overview of the category $\mathcal{LM}$ as def\/ined by Loday and Pirashvili~\cite{Loday98}. In Section~\ref{der LM} we give the necessary tools and background to describe the universal algebra of derivations of an algebra in $\mathcal{LM}$ (see Proposition~\ref{universal Der LM}). Lastly, in Section~\ref{LR in LM}, we describe Lie--Rinehart algebras in the  category~$\mathcal{LM}$ of linear maps.

\subsection[The category $\mathcal{LM}$ of linear maps]{The category $\boldsymbol{\mathcal{LM}}$ of linear maps}

We f\/irst recall some fundamental concepts and  def\/initions about the category $\mathcal{LM}$ of linear maps, introduced by Loday and Pirashvili in~\cite{Loday98}, that are relevant for our main constructions later. We refer to~\cite{Loday98} for further details. See~\cite{UliWagemann} for results on Hopf algebras in~$\mathcal{LM}$.

\begin{Definition}
The objects in the category $\mathcal{LM}$  are $R$-module maps  $(V \overset{ u}{\rightarrow} W) $, where $u$ is called the vertical map. The morphisms  between objects in $\mathcal{LM}$ are pairs of maps $\mathsf h := ( h_1, h_0)$ such that the following diagram commutes:
\begin{gather*}
 \xymatrixcolsep{3.5pc}
 \xymatrixrowsep{2.2pc}
    \xymatrix{
  V \ar[r]^{h_1}    \ar[d]_u & V'  \ar[d]^{u'} \\
        W  \ar[r]^{h_0}  &   W'   }
\end{gather*}
Given two morphisms $\mathsf g := ( g_1 , g_0)$  and $\mathsf h := ( h_1 , h_0)$  in $\mathcal{LM}$, their composition $\mathsf h \circ \mathsf g$
is given by
\begin{gather}
\label{hgcomp}
\mathsf h \circ \mathsf g = ( h_1 , h_0 ) \circ ( g_1 , g_0 ) : = ( h_1 \circ g_1 , h_0 \circ g_0 ).
\end{gather}
\end{Definition}

A morphism $\upphi := (\phi_1 , \phi_0)$ is an isomorphism between objects $(V \overset{ u}{\rightarrow} W)$ and $(V' \overset{u'}{\rightarrow} W')$ if and only if $\phi_1$ and $\phi_2$ are isomorphisms of $R$-modules.

\begin{Proposition}
The category $\mathcal{LM}$ is monoidal where the tensor product of two objects  is
\begin{gather*}
 \big(V \overset{ u}{\rightarrow} W\big)  \otimes \big( V' \overset{u'}{\rightarrow} W'\big) := \big(   V \otimes W' \oplus W \otimes  V'  \xlongrightarrow{u \otimes 1_{W'} + 1_W \otimes u'}   W \otimes W' \big),
\end{gather*}
and the unit object is $( \{ 0 \} \overset{0}{\rightarrow} R)$. Moreover, given two morphisms  $\mathsf g := ( g_1 , g_0)$ and $ \mathsf h := ( h_1 , h_0)$ in $\mathcal{LM}$, their tensor product $\mathsf g \otimes \mathsf h$
  is given by
\begin{gather}
\label{hgtensor}
\mathsf g \otimes \mathsf h = ( g_1 , g_0 ) \otimes ( h_1 , h_0 ) := ( g_1 \otimes h_0 + g_0 \otimes h_1  , g_0 \otimes  h_0 ).
\end{gather}
\end{Proposition}

\begin{proof}
See \cite{Loday98}.
\end{proof}

Furthermore, the monoidal category $\mathcal{LM}$ is symmetric, with interchange morphism  denoted by  $\uptau \colon  (V \overset{ u}{\rightarrow} W) \otimes ( V' \overset{u'}{\rightarrow} W')  \rightarrow  ( V' \overset{u'}{\rightarrow} W') \otimes (V \overset{ u}{\rightarrow} W)$ and given by   $\tau_0 \colon  W \otimes W' \rightarrow W' \otimes W$ and $\tau_1 \colon  V \otimes W' \oplus W \otimes V' \rightarrow V' \otimes W \oplus W' \otimes V$, see  \cite{Loday98} for more details.

Commutative diagrams in $\mathcal{LM}$ can be seen as commutative ``cubes" in the category  $R$-Mod.
\begin{Example}
The commutative diagram in $\mathcal{LM}$ given by
\begin{gather*}
\begin{gathered}
 \xymatrixcolsep{2.5pc}
 \xymatrixrowsep{2.2pc}\xymatrix{
  (V \overset{ u}{\rightarrow} W) \ar[r]^{\upmu'}  \ar[d]_{\upvarphi'} & (V' \overset{u'}{\rightarrow} W') \ar[d]^\upvarphi \\
   ( M' \overset{g'}{\rightarrow} A')  \ar[r]^{\upmu} &  (M \overset{ g}{\rightarrow} A)  }
\end{gathered}
 \end{gather*}
corresponds to the commuting ``cube" given by
\begin{gather*}
\begin{gathered}
 \xymatrixcolsep{2.8pc}
 \xymatrixrowsep{2.2pc}
 \xymatrix@!0{
& V \ar@{->}[rr]^-{\mu_1'}  \ar@{->}'[d]^u[dd]
&& V'  \ar@{->}[dd]^{u'} \\
M' \ar@{<-}[ur]^-{ \varphi'_1 } \ar@{->}[rr]^{\mu_1 \ \ \ \ \ \ \ } \ar@{->}[dd]_{g'}
&& M \ar@{<-}[ur]_-{\varphi_1} \ar@{->}'[d]^g [dd] \\
&   W   \ar@{->}[rr]^-{ \ \ \ \ \ \mu_0' }
&& W' \\
A' \ar@{->}[rr]_{\mu_0} \ar@{<-}[ur]^{\varphi_0'}
&& A  \ar@{<-}[ur]_{ \varphi_0}
}
\end{gathered}
\end{gather*}
\end{Example}

We now describe some  of the  fundamental algebraic structures in $\mathcal{LM}$. For further details and proofs see \cite{Loday98}.

\begin{Proposition}\quad

\begin{itemize}\itemsep=0pt

\item
An \textbf{associative algebra object} $ ( M \overset{ g}{\rightarrow} A)$ in $\mathcal{LM}$ is a triple consisting of an associative $R$-algebra  $A$, an $A$-bimodule $M$ and an $A$-bimodule map $g \colon  M \rightarrow A$. Moreover, the algebra object  $(M \overset{ g}{\rightarrow} A)$ is commutative, if and only if the $A$-bimodule $M$ is symmetric and $A$ is commutative.

\item
A \textbf{Lie algebra object} $( N \overset{ f}{\rightarrow} L)$  in $\mathcal{LM}$ is equivalent to a Lie algebra $ L$,  a right $L$-mo\-du\-le~$N$ with  right action $N \otimes L \rightarrow N$ given by $n  \otimes \xi \mapsto [ n, \xi]$ for all $n \in N$ and $\xi \in L$, and an $R$-linear $L$-equivariant  map $f\colon  N \rightarrow L$, i.e., $f ( [ n , \xi ]) = [ f(n) , \xi ]_L$.
\end{itemize}
\end{Proposition}

\begin{Example}
We give the following examples of objects in $\mathcal{LM}$:
\begin{itemize}\itemsep=0pt
\item The  surjective map $\pi \colon  \mathcal E \rightarrow \mathcal E_{\mathrm{Lie}}$ is a Lie algebra object in $\mathcal{LM}$.

\item  Let $I$ be a two-sided ideal in an associative algebra $A$. The identity map $\mathrm{id}\colon  I \rightarrow A$ is an associative algebra in $\mathcal{LM}$.

\item  Let $B$ be the square-zero extension of an associative algebra $A$ by the $A$-module $M$. Then $( M \hookrightarrow B)$ is an algebra object in $\mathcal{LM}$.

\item Let $A$ be a Poisson algebra with bracket $\{ - , - \}$, and let $\Omega^1 (A)$ be the $A$-module of K\"ahler dif\/ferentials over $A$. The pair $( A , \Omega^1(A) )$ is a Lie--Rinehart algebra with anchor map $\rho\colon  \Omega^1 (A) \rightarrow \mathrm{Der}(A)$ given by $da \mapsto \{ a , - \}$ for all $a \in A$ and Lie algebra structure on~$\Omega^1 (A)$ given by $[ da , db ]_{\Omega^1 (A)} := d \{ a , b \}$. The dif\/ferential map $d \colon  A \rightarrow \Omega^1 (A)$ is a Lie algebra object in $\mathcal{LM}$, where $A$ is a right $\Omega^1 (A)$-module with action $A \otimes \Omega^1 (A) \rightarrow A$ given by $a \otimes b \cdot dc \mapsto b \cdot  \{ a , c  \}$ for all $a , b, c  \in A$.

\item Similarly, let  $A$ be a Jacobi algebra with  bracket $\{ - , - \}_J$, and let $\mathcal J^1 (A)$ be its 1-jet space. Then the pair $(A , \mathcal J^1 (A) )$ is a Lie--Rinehart algebra (see \cite{R2} for more details), and the map $j \colon  A \rightarrow \mathcal J^1 (A)$ given by $a \mapsto j^1(a)$ is a Lie algebra object in $\mathcal{LM}$ with right $\mathcal J^1 (A)$-action
on $A$ given by $a \otimes b \cdot j^1 (c) \mapsto b \cdot  \{ a , c  \}_J$ for all $a , b, c  \in A$.
\end{itemize}
\end{Example}

Note that a Lie algebra object  $( N \overset{ f}{\rightarrow} L)$  in $\mathcal{LM}$ is a very similar object to a strict 2-term~$L_\infty$ algebra.

We now focus on the description of   $( M \overset{g}{\rightarrow} A)$-modules in $\mathcal{LM}$: 

\begin{Proposition}
\label{left module}
   A \textbf{left $( M \overset{ g}{\rightarrow} A)$-module object}  is  a  map $(V \overset{ u}{\rightarrow} W)$ of  left $A$-modules such that there exists an $A$-module map $\mu_1 \colon  W \otimes M \rightarrow V$ satisfying $g \circ  \mu_1 ( m \otimes w ) = \mu_0 ( g(m) \otimes w ) $ for $w \in W $ and $m \in M$, which descends to an $A$-module map  $ \alpha_\ell^V\colon  M \otimes_A W \rightarrow V $ $($called \textbf{structure map} of the left $( M \overset{g}{\rightarrow} A)$-module $(V \overset{ u}{\rightarrow} W))$  satisfying
\begin{gather}
\label{alpha ell}
 u \circ \alpha_\ell^V ( m \otimes_A w ) =   \mu_0 \circ ( g(m)  \otimes_A w  ) .
\end{gather}
\end{Proposition}

\begin{proof}Since $(V \overset{ u}{\rightarrow} W)$  is a left $( M \overset{g}{\rightarrow} A)$-module, there exists a morphism
\begin{gather*}
\upmu\colon  \ \big ( M \overset{g}{\rightarrow} A\big) \otimes \big(V \overset{ u}{\rightarrow} W\big) \longrightarrow \big(V \overset{ u}{\rightarrow} W\big),
\end{gather*}
that is, a commuting square
\begin{gather}
\label{module square 1}
\begin{gathered}
 \xymatrixcolsep{3.5pc}
 \xymatrixrowsep{2.5pc}
    \xymatrix{
M \otimes W \oplus A \otimes V  \ar[r]^-{\mu_1}    \ar[d]_{g \otimes 1_W + 1_A \otimes u} & V  \ar[d]^{u} \\
       A \otimes W  \ar[r]^{\mu_0}  &  W  }
\end{gathered}
\end{gather}
where $\mu_1$ and $\mu_0$ satisfy some associativity conditions.  From \eqref{module square 1} we see that  $\mu_0 $ and  the restriction of $\mu_1$ to $A \otimes V$ turn $W$ and $V$ respectively into left $A$-modules, so that the vertical map $u$ becomes a map of left $A$-modules. Now, since $M$ is an $A$-bimodule and $W$ is a left $A$-module, we can construct the tensor product $M \otimes_A W$. Moreover, by the associativity of the module action $\upmu$,    we deduce that  $\mu_1$ vanishes on $m \cdot a \otimes w - m \otimes a \cdot w$ where $m \in M$, $a \in A$, $w \in W$ so that the map  $ \mu_1 \colon  M \otimes W \rightarrow V $  descends to a map $ \alpha_\ell^V\colon  M \otimes_A W \rightarrow V $ yielding the following  diagram:
\begin{gather}
\label{module square 2}
\begin{gathered}
 \xymatrixcolsep{3.5pc}
 \xymatrixrowsep{2.5pc}
    \xymatrix{
M \otimes_A W  \ar[r]^-{\alpha_\ell^V}    \ar[d]_{g \otimes_A 1_W } & V  \ar[d]^{u} \\
       A \otimes_A W  \ar[r]^{\mu_0}  &  W  }
\end{gathered}
\end{gather}
The commutativity of \eqref{module square 2} ensures that  the compatibility relation \eqref{alpha ell} is satisf\/ied.
\end{proof}

\begin{Proposition}
For any Lie algebra object $( N \overset{ f}{\rightarrow} L)$ endowed with an $( M \overset{g}{\rightarrow} A)$-module structure, the Leibniz bracket on $N$ given  by
$  [ n_1 , n_2 ]_N := [ n_1 , f(n_2) ]$
satisfies
\begin{gather}
\label{extra}
[ n_1 , a \cdot [ n_2 , n_2 ]_N ]_N = 0, \qquad \forall \, a \in A.
\end{gather}
\end{Proposition}

\begin{proof}
By def\/inition, a Lie algebra object $( N \overset{ f}{\rightarrow} L)$ carries a Leibniz algebra structure on the $L$-module $N$ with bracket given by $ [ n_1 , n_2 ]_N := [ n_1 , f(n_2) ]$ so that  $[ n_1 , n_2 ]_N =0 $ for all $n_2 \in \mathrm{Ker}( f)$.     Now, since $( N \overset{ f}{\rightarrow} L)$ is an $( M \overset{g}{\rightarrow} A)$-module, $N$ and $L$ are $A$-modules and the vertical map $f$ is $A$-linear, we deduce that  $\mathrm{Ker} ( f)$ is an $A$-module. Hence $[ - , - ]_N$ satisf\/ies condition \eqref{extra}.
\end{proof}

\subsection[The Lie algebra of derivations in $\mathcal{LM}$]{The Lie algebra of derivations in $\boldsymbol{\mathcal{LM}}$}
\label{der LM}

In this section we describe the Lie algebra of derivations of an associative algebra object in  $\mathcal{LM}$. We start by describing  morphisms between Lie algebra objects  and Lie algebra actions.

\begin{Proposition}
\label{prop algebra map}
An algebra morphism $\upphi \colon  ( M \overset{ g}{\rightarrow} A) \rightarrow ( M' \overset{g'}{\rightarrow} A') $ in $\mathcal{LM}$ is given by a pair of maps $( \phi_1 , \phi_0 )$ satisfying
\begin{gather}
 \phi_1 ( a_1 m_1 + m_2 a_2) = \phi_0 ( a_1) \phi_1 ( m_1) + \phi_1 ( m_2) \phi_0 ( a_1), \label{algebra map 1} \\
 \phi_0 ( a_1 \cdot a_2 ) = \phi_0 ( a_1 ) \cdot \phi_0 ( a_2) \label{algebra map 0}
\end{gather}
for $a_1 , a_2 \in A$ and $m_1 , m_2 \in M$.
\end{Proposition}

\begin{proof} Assume $\upphi := (\phi_1 , \phi_0)$ is an algebra morphism (in $\mathcal{LM}$). Then using~\eqref{hgcomp} and~\eqref{hgtensor}, a~straightforward computation shows
\begin{gather*}
0  =    \upphi \circ \upmu - \upmu \circ ( \upphi \otimes \upphi)
 = (\phi_1 \circ \mu_1  - \mu_1 \circ ( \phi_1 \otimes \phi_0 ) , \phi_0 \circ \mu_0 - \mu_0 \circ ( \phi_0 \otimes \phi_0 ) ),
\end{gather*}
that is
\begin{gather}
\label{algebra map commutative}
\phi_1 \circ \mu_1 - \mu_1 \circ ( \phi_1 \otimes \phi_0) = 0, \qquad \phi_0 \circ \mu_0 - \mu_0 \circ ( \phi_0 \otimes \phi_0 ) = 0
\end{gather}
which yield the relations in \eqref{algebra map 1} and \eqref{algebra map 0}.
\end{proof}

\begin{Proposition}
A Lie algebra map $\mathbf{a} \colon  ( N \overset{ f}{\rightarrow} L)\rightarrow ( N' \overset{f'}{\rightarrow} L')$ in $\mathcal{LM}$ is given by a pair of maps $(a_1 , a_0)$ satisfying
\begin{gather}
\label{lie algebra map}
 a_1 ( [ n ,  \xi ]) = [  a_1 (n)  , a_0 ( \xi ) ] , \qquad   a_0 ( [ \xi, \zeta ]_L ) = [ a_0 ( \xi ) , a_0 ( \zeta )]_{L'}
\end{gather}
for $n \in N$ and $\xi , \zeta \in L$.
\end{Proposition}

\begin{proof}
An identical argument as in the proof for Proposition~\ref{prop algebra map} and abusing notation so that~$\upmu$ is the Lie bracket on $( N \overset{ f}{\rightarrow} L)$ we obtain the relation in~\eqref{algebra map commutative} which, in this case, yields the relations in~\eqref{lie algebra map}.
\end{proof}

\begin{Proposition}
\label{3maps prop}
Given  a Lie algebra object $( N \overset{ f}{\rightarrow} L)$, a {left $( N \overset{ f}{\rightarrow} L)$-module} $($in $\mathcal{LM})$  is an object~$(V \overset{ u}{\rightarrow} W)$  such that $V$ and $W$ are  left $L$-modules with actions given  by
\begin{gather}
\label{3maps}
 \alpha_0\colon  \ L \otimes W \longrightarrow  W, \qquad  \alpha_2 \colon  \  L \otimes V \longrightarrow V,
\end{gather}
and there exists an $R$-linear  map
\begin{gather}
\label{third map}
 \alpha_1 \colon \ N \otimes W \longrightarrow V
\end{gather}
 satisfying the following  compatibility condition
\begin{gather}
\label{compatibility3}
\alpha_1 ( [ n , \xi ] \otimes w ) = \alpha_1 ( n \otimes \alpha_0 ( \xi \otimes w ) ) - \alpha_2 ( \xi \otimes  \alpha_1 ( n \otimes w ) ) .
\end{gather}
Moreover, the following compatibility conditions between $u$, $f$, $\alpha_1$, $\alpha_2$ and $\alpha_3$ are  satisfied
\begin{gather*}
u \circ  \alpha_1  = \alpha_0 \circ ( f \otimes 1_W ) , \qquad u \circ  \alpha_2   =  \alpha \circ ( 1_L \otimes u).
\end{gather*}
\end{Proposition}

\begin{proof}
Since $(V \overset{ u}{\rightarrow} W)$ is a left $( N \overset{ f}{\rightarrow} L)$-module, we have  a morphism
\begin{gather*}
\upalpha\colon  \ \big( N \overset{ f}{\rightarrow} L\big) \otimes  \big(V \overset{ u}{\rightarrow} W\big) \longrightarrow \big(V \overset{ u}{\rightarrow} W\big)
\end{gather*}
that is, a pair of maps $( \alpha_1 + \alpha_2 , \alpha_0)$  where   $\alpha_0$, $\alpha_1$, $\alpha_2$ are given by the maps in \eqref{3maps} and \eqref{third map}, such that the  diagram commutes
\begin{gather}
\label{comm3maps}
\begin{gathered}
 \xymatrixcolsep{9pc}
\xymatrixrowsep{2.7pc}
 \xymatrix{
( N \overset{ f}{\rightarrow} L) \otimes ( N \overset{ f}{\rightarrow} L)  \otimes (V \overset{ u}{\rightarrow} W)    \ar[r]^-{\mathsf{id} \otimes \upalpha -  ( \mathsf{id} \otimes \upalpha  ) \circ ( \uptau \otimes \mathsf{id})  }    \ar[d]_{ \upmu \otimes \mathsf{id} } & \left(    ( N \overset{ f}{\rightarrow} L) \otimes (V \overset{ u}{\rightarrow} W) \right)^{\oplus 2}      \ar[d]^\upalpha \\
( N \overset{ f}{\rightarrow} L) \otimes (V \overset{ u}{\rightarrow} W) \ar[r]^-{\upalpha}   & (V \overset{ u}{\rightarrow} W) }
\end{gathered}\hspace{-10mm}
\end{gather}
which can be seen as the following diagram in cube shape
\begin{gather*}
\begin{gathered}
 \xymatrixcolsep{8pc}
 \xymatrixrowsep{2.7pc}
 \xymatrix@!0{
& N \otimes L \otimes W + L \otimes N \otimes W + L \otimes L \otimes V \ar@{->}[rr]^-\beta  \ar@{->}'[d][dd]
&& ( N \otimes W + L \otimes V)^{ \oplus 2}  \ar@{->}[dd]^{f \otimes 1_W + 1_L \otimes u} \\
N \otimes W + L \otimes V \ar@{<-}[ur]^-{ \mu_1 \otimes 1_W + \mu_0 \otimes 1_V \ \ \ \  \ } \ar@{-}[rr]^{\ \ \ \ \ \ \ \ \ \ \ \ \ \ \ \ \ \ \ \ \ \ \alpha_1 + \alpha_2} \ar@{->}[dd]_{f \otimes 1_W + 1_L \otimes u}
&& V \ar@{<-}[ur]_-{\alpha_1 + \alpha_2} \ar@{->}'[d]^u[dd] \\
&  L \otimes L \otimes W   \ar@{->}[rr]^-{\mathrm{id}_0 \otimes \alpha_0 - (\mathrm{id}_0 \otimes \alpha_0) \circ (\tau_0 \otimes \mathrm{id}_0) }
&& (  L \otimes W  )^{ \oplus 2}\\
L \otimes W \ar@{->}[rr]_{\alpha_0} \ar@{<-}[ur]^{\mu_0 \otimes 1_W}
&& W  \ar@{<-}[ur]_{\alpha_0}
}
\end{gathered}
\end{gather*}
where
\begin{gather*}
\beta := \mathrm{id}_1 \otimes \alpha_0  -  (\mathrm{id}_0 \otimes \alpha_1  ) \circ ( \tau_1 \otimes  \mathrm{id}_0 )  + \mathrm{id}_0 \otimes \alpha_1\\
  \hphantom{\beta :=}{} -  (  \mathrm{id}_1 \otimes \alpha_0  ) \circ ( \tau_1 \otimes \mathrm{id}_0 ) +    \mathrm{id}_0 \otimes \alpha_2 -  ( \mathrm{id}_0 \otimes \alpha_2  ) \circ ( \tau_0 \otimes \mathrm{id}_1).
\end{gather*}

Using  \eqref{hgcomp} and \eqref{hgtensor}, a long but straightforward computation  shows that the compatibility relation making \eqref{comm3maps} commute can be expressed as
\begin{gather*}
0   = \upalpha \circ (\mathsf{id} \otimes \upalpha   -  ( \mathsf{id} \otimes \upalpha ) \circ ( \uptau \otimes \mathsf{id})  ) - \upalpha \circ (\upmu \otimes \mathsf{id} ) \\
\hphantom{0}{} = \big( \alpha_1 \circ ( \mathrm{id}_1 \otimes \alpha_0 ) - \alpha_2 \circ ( \mathrm{id}_0 \otimes \alpha_1 ) \circ ( \tau_1 \otimes \mathrm{id}_0 )  + \alpha_2 \circ ( \mathrm{id}_0 \otimes \alpha_1 )  \\
\hphantom{0=}{} - \alpha_1 \circ (\mathrm{id}_1 \otimes \alpha_0 ) \circ ( \tau_1 \otimes \mathrm{id}_0 ) + \alpha_2 \circ ( \mathrm{id}_0 \otimes \alpha_2 ) - \alpha_2 \circ ( \mathrm{id}_0 \otimes \alpha_2 ) \circ ( \tau_0 \otimes \mathrm{id}_1 ) , \\
\hphantom{0=}{} \quad \alpha_0 \circ ( \mathrm{id}_0 \otimes \alpha_0 ) - \alpha_0 \circ ( \mathrm{id}_0 \otimes \alpha_0 ) \circ ( \tau_0 \otimes \mathrm{id}_0 )  \big) \\
\hphantom{0=}{} - \big( \alpha_1 \circ ( \mu_1 \otimes \mathrm{id}_0 ) + \alpha_2 \circ ( \mu_0 \otimes \mathrm{id}_1 ) , \alpha_0 \circ (\mu_0 \otimes \mathrm{id}_0 ) \big) \\
\hphantom{0}{} = \big( \alpha_1 \circ ( \mathrm{id}_1 \otimes \alpha_0 ) - \alpha_2 \circ ( \mathrm{id}_0 \otimes \alpha_1 ) \circ ( \tau_1 \otimes \mathrm{id}_0 )  - \alpha_1 \circ ( \mu_1 \otimes \mathrm{id}_0 ) \\
\hphantom{0=}{} + \big( \alpha_2 \circ ( \mathrm{id}_0 \otimes \alpha_1 ) - \alpha_1 \circ ( \mathrm{id}_1 \otimes \alpha_0 ) \circ ( \tau_1 \otimes \mathrm{id}_0 ) - \alpha_1 \circ ( \mu_1 \otimes \mathrm{id}_0 ) \\
\hphantom{0=}{} + \big( \alpha_2 \circ ( \mathrm{id}_0 \otimes \alpha_2 ) - \alpha_2 \circ ( \mathrm{id}_0 \otimes \alpha_2 ) \circ ( \tau_0 \otimes \mathrm{id}_1 ) - \alpha_2 \circ ( \mu_0 \otimes \mathrm{id}_1 ) , \\
\hphantom{0=}{}  \quad \alpha_0 \circ ( \mathrm{id}_0 \otimes \alpha_0 ) - \alpha_0 \circ ( \mathrm{id}_0 \otimes \alpha_0 ) \circ ( \tau_0 \otimes \mathrm{id}_0 )  - \alpha_0 \circ (\mu_0 \otimes \mathrm{id}_0 )  \big),
\end{gather*}
so that the following diagrams commute
\begin{itemize}\itemsep=0pt
 \item A  diagram encoding the $L$-module action $  \alpha_0 \colon  L \otimes W \rightarrow W$
\begin{gather*}
\begin{gathered}
 \xymatrixcolsep{11pc}
\xymatrixrowsep{2.5pc}    \xymatrix{
L \otimes L \otimes W     \ar[r]^-{ \mathrm{id}_0 \otimes \alpha_0 -  ( \mathrm{id}_0 \otimes \alpha_0  ) \circ ( \tau_0 \otimes \mathrm{id}_0)  }    \ar[d]_{ \mu_0 \otimes {1_W} } &   L \otimes W \oplus L \otimes W          \ar[d]^{\alpha_0}  \\
L  \otimes W \ar[r]^-{\alpha_0}   & W }
\end{gathered}
\end{gather*}

\item
A diagram encoding the $L$-module action $\alpha_2\colon  L \otimes V \rightarrow V$
\begin{gather*}
\begin{gathered}
 \xymatrixcolsep{11pc}
\xymatrixrowsep{2.5pc}   \xymatrix{
L  \otimes L \otimes V   \ar[r]^-{ \mathrm{id}_0 \otimes \alpha_2 -  ( \mathrm{id}_0 \otimes \alpha_2  ) \circ ( \tau_0 \otimes \mathrm{id}_1)  }    \ar[d]_{ \mu_0 \otimes 1_V } & L \otimes V \oplus  L \otimes V      \ar[d]^{\alpha_2} \\
 L \otimes V \ar[r]^-{\alpha_2}   & V }
\end{gathered}
\end{gather*}
\item
And lastly
 \begin{gather*}
\begin{gathered}
 \xymatrixcolsep{11pc}
\xymatrixrowsep{2.5pc}  \xymatrix{
N \otimes L \otimes W + L \otimes N \otimes W    \ar[r]^-{\substack{ \mathrm{id}_1 \otimes \alpha_0  -  ( \mathrm{id}_0 \otimes \alpha_1  ) \circ ( \tau_1 \otimes  \mathrm{id}_0 )  \\+  \mathrm{id}_0 \otimes \alpha_1  -  (  \mathrm{id}_1 \otimes \alpha_0  ) \circ ( \tau_1 \otimes \mathrm{id}_0 ) } }    \ar[d]_{ \mu_1 \otimes {1_W}  } & ( N \otimes W \oplus  L \otimes V )^{ \oplus 2}     \ar[d]^{\alpha_1 + \alpha_2} \\
N \otimes W  \ar[r]^-{\alpha_1 }   & V }
\end{gathered}
\end{gather*}
encoding the compatibility relation in \eqref{third map}.\hfill $\qed$
\end{itemize}
\renewcommand{\qed}{}
\end{proof}

By the adjoint functor property of tensor products, the maps in \eqref{3maps} correspond to
\begin{gather*}
\alpha_0 \colon \ L \rightarrow \mathrm{Hom}_R (W, W) , \qquad \alpha_1 \colon  \ N \rightarrow \mathrm{Hom}_R ( W, V), \qquad \alpha_2 \colon  \ L \rightarrow \mathrm{Hom}_R (V , V)
\end{gather*}
that we can describe as the following commutative diagram
\begin{gather}
\label{diag hom}
\begin{gathered}
 \xymatrixcolsep{5pc} \xymatrixrowsep{2.5pc}  \xymatrix{
  N \ar[r]^-{\alpha_1}   \ar[d]_f   &  \mathrm{Hom}_R ( W , V)  \ar[d]^{ g \circ h + h \circ g} \\
         L  \ar[r]^-{\alpha_0 + \alpha_2   }   &  \mathrm{Hom}_R ( W,W) \oplus \mathrm{Hom}_R ( V,V)  }
\end{gathered}
\end{gather}
where $h  \in \mathrm{Hom}_R ( W , V)$.

\begin{Proposition}
\label{der action}
Let $(M \overset{ g}{\rightarrow} A)$ be a commutative  algebra object in $\mathcal{LM}$. A Lie algebra object $( N \overset{ f}{\rightarrow} L)$ is said to \textbf{act on $( M \overset{g}{\rightarrow} A)$ by derivations} if there exist two Lie algebra maps
\begin{gather}
\label{tripleAction}
\rho_0 \colon  \ L  \rightarrow \mathrm{Der}_R(A),\qquad  \rho_2 \colon \ L  \rightarrow H:= ( \mathrm{Hom}_R (M,M) , [ - , - ]_H )
\end{gather}
satisfying the compatibility conditions
\begin{gather}\label{compDer1}
 \rho_2 ( \xi) ( a \cdot m) = a \cdot \rho_2 (\xi)(m) + \rho_0( \xi ) (a) \cdot m , \quad   g \left( \rho_2( \xi ) ( m) \right)  = \rho_0 ( \xi ) (g (m))
\end{gather}
and an $R$-module map
\begin{gather*}
 \rho_1 \colon  \ N  \rightarrow \mathrm{Der}_R (A ,M)
\end{gather*}
satisfying
\begin{gather}
\rho_1 ( [ n , \xi ]) = [\rho_1 (n) , ( \rho_0 + \rho_2 )( \xi ) ], \qquad
  g ( \rho_1 (n)  (a) ) = \rho_0 ( f(n) ) (a) \label{compDer3}
\end{gather}
for all $\xi \in L$, $a \in A$ and $m \in M$.
\end{Proposition}

\begin{proof}
Let  a Lie algebra $( N \overset{ f}{\rightarrow} L)$ act  on $(M \overset{ g}{\rightarrow} A)$ by derivations, then there exists a left $( N \overset{ f}{\rightarrow} L)$-module structure on $(M \overset{ g}{\rightarrow} A)$, denoted  by $\upvarrho\colon  ( N \overset{ f}{\rightarrow} L) \otimes (M \overset{ g}{\rightarrow} A) \rightarrow (M \overset{ g}{\rightarrow} A)$, which  by Proposition \ref{3maps prop}  endows  $M$ and $A$ with left $L$-actions given by   $ \varrho_0 \colon  L \otimes A \rightarrow A$ and  $\varrho_2\colon   L \otimes M \rightarrow M$ respectively, and induces a map $ \varrho_1 \colon  N \otimes A \rightarrow M$ satisfying \eqref{compatibility3}. Furthermore, the action of $( N \overset{ f}{\rightarrow} L)$ on $(M \overset{ g}{\rightarrow} A)$ by derivations makes  the following  commute
\begin{gather*}
\begin{gathered}
 \xymatrixcolsep{8.5pc}
\xymatrixrowsep{1.9pc}   \xymatrix{
( N \overset{ f}{\rightarrow} L) \otimes (M \overset{ g}{\rightarrow} A)  \otimes (M \overset{ g}{\rightarrow} A)  \ar[r]^-{\upvarrho \otimes \mathsf{id} + ( \mathsf{id} \otimes \upvarrho  ) \circ ( \uptau \otimes \mathsf{id}  )  }    \ar[d]_{ \mathsf{id} \otimes \upmu} & (M \overset{ g}{\rightarrow} A)^{\otimes 2} \oplus (M \overset{ g}{\rightarrow} A)^{\otimes 2}\ar[d]^\upmu \\
( N \overset{ f}{\rightarrow} L) \otimes (M \overset{ g}{\rightarrow} A) \ar[r]^-{\upvarrho}   & (M \overset{ g}{\rightarrow} A) }
\end{gathered}
\end{gather*}
Since $\uprho = ( \varrho_1 + \varrho_2 , \varrho_0 ) $,  by \eqref{hgcomp} and \eqref{hgtensor} we f\/ind
 \begin{gather*}
0  = \upmu \circ \big( \uprho \otimes \mathsf{id} + ( \mathsf{id} \otimes \uprho ) \circ ( \uptau \otimes \mathsf{id} ) \big) - \uprho \circ ( \mathsf{id} \otimes \upmu ) \\
\hphantom{0}{} =  \big( \mu_1 \circ ( \varrho_1 \otimes \mathrm{id}_0 ) + \mu_1 \circ ( \mathrm{id}_0 \otimes  \varrho_1  ) \circ ( \tau_0 \otimes \mathrm{id}_1  ) - \rho_1 \circ ( \mathrm{id}_1 \otimes \mu_0 )+ \mu_1 \circ ( \varrho_2 \otimes \mathrm{id}_0 )   \\
\hphantom{0=}{}       + \mu_1 \circ (\mathrm{id}_1 \otimes \varrho_0 ) \circ ( \tau_1 \otimes  \mathrm{id}_0  )   + \mu_1 \circ ( \varrho_0 \circ \mathrm{id}_1 ) + \mu_1 \circ ( \mathrm{id}_0 \otimes \varrho_2   ) \circ ( \tau_1 \otimes   \mathrm{id}_0 )   \\
\hphantom{0=}{} - \varrho_2 \circ ( \mathrm{id}_1 \otimes \mu_0 )   ,           \mu_0 \circ (  \varrho_0 \otimes \mathrm{id}_0 )  +  \mu_0 \circ ( \mathrm{id}_0 \otimes  \varrho_0 ) \circ (  \tau_0 \otimes \mathrm{id}_0 )   - \varrho_0 \circ ( \mathrm{id}_0 \otimes \mu_0 )  \big),
\end{gather*}
so the following diagrams commute:
\begin{itemize}\itemsep=0pt
\item  a diagram which encodes the universal action of a Lie algebra $L$ on $A$ by derivations
\begin{gather*}
 \xymatrixcolsep{11pc}
\xymatrixrowsep{1.8pc}   \xymatrix{
   L \otimes A \otimes A \ar[r]^-{\varrho_0 \otimes \mathrm{id}_0 + ( \mathrm{id}_0 \otimes \varrho_0  ) \circ (  \tau_0 \otimes \mathrm{id}_0 ) }    \ar[d]_{ 1_L \otimes \mu_0} & ( A \otimes A )^{ \oplus 2 }  \ar[d]^{\mu_0} \\
         L \otimes A \ar[r]^-{\varrho_0}  &A  }
\end{gather*}

\item a diagram encoding the action of $N$ on $A$
\begin{gather*}
  \xymatrixcolsep{11pc}
\xymatrixrowsep{1.8pc}  \xymatrix{
  N \otimes A \otimes A  \ar[r]^-{\varrho_1 \otimes \mathrm{id}_0 + ( \mathrm{id}_1  \otimes  \varrho_1  )  \circ \left( \tau_0 \otimes  \mathrm{id}_0  \right)  }    \ar[d]_{1_N \otimes \mu_0 } & ( M \otimes A \oplus A \otimes M )^{ \oplus 2}  \ar[d]^{\mu_1} \\
         N \otimes A  \ar[r]^-{\varrho_1} &M  }
\end{gather*}

\item and lastly, a commutative diagram encoding the action of $L$ on both $A$ and $M$
\begin{gather*}
\begin{gathered}
  \xymatrixcolsep{11pc}
  \xymatrixrowsep{1.8pc}
  \xymatrix{
    L \otimes M \otimes A \oplus  L \otimes A \otimes M  \ar[r]^-{ \substack{ \varrho_2 \otimes \mathrm{id}_0  + ( \mathrm{id}_1 \otimes \varrho_0) \circ ( \tau_1 \otimes \mathrm{id}_0 ) \\ + \varrho_0 \otimes \mathrm{id}_1  + ( \mathrm{id}_0 \otimes \varrho_2  ) \circ ( \tau_0 \otimes \mathrm{id}_1  )} } \ar[d]_{1_L \otimes \mu_1 } &  ( M \otimes A \oplus A \otimes M )^{ \oplus   2}  \ar[d]^{\mu_1} \\
         L \otimes M \ar[r]^-{\varrho_2}  & M  }
\end{gathered}
\end{gather*}
\end{itemize}
These maps make the following cube commute
\begin{gather*}
\begin{gathered}
 \xymatrixcolsep{8pc}
 \xymatrixrowsep{2.25pc}
\xymatrix@!0{
& N \otimes A \otimes A + L \otimes M \otimes A + L \otimes A \otimes M \ar@{->}[rr]^-\beta  \ar@{->}'[d][dd]
&& ( M \otimes A  \oplus A  \otimes M)^{ \oplus 2}  \ar@{->}[dd]^{g \otimes 1_A + 1_M \otimes g} \\
N \otimes A + L \otimes M \ar@{<-}[ur]^-{ \mu_1 \otimes 1_A + \mu_0 \otimes 1_M \ \ \ \  \ } \ar@{-}[rr]^{\ \ \ \ \ \ \ \ \ \ \ \ \ \ \ \ \ \ \ \ \ \ \varrho_1 + \varrho_2} \ar@{->}[dd]_-{f \otimes 1_A + 1_L \otimes g}
&& M \ar@{<-}[ur]_-{\mu_1} \ar@{->}'[d]^-g[dd] \\
&  L \otimes A \otimes A   \ar@{->}[rr]^-{ \varrho_0 \otimes \mathrm{id}_0 - (\mathrm{id}_0 \otimes \varrho_0) \circ (\tau_0 \otimes  \mathrm{id}_0 ) }
&& (  A \otimes A )^{ \oplus 2}\\
L \otimes A \ar@{->}[rr]_{\varrho_0} \ar@{<-}[ur]^{ 1_L \otimes \mu_0 }
&& A  \ar@{<-}[ur]_{\mu_0}
}
\end{gathered}
\end{gather*}
By the adjoint functor property of tensor products, the maps $\varrho_1$, $\varrho_2$ and $ \varrho_3$ are equivalent to the maps in \eqref{tripleAction}.
\end{proof}

\begin{Proposition}
\label{universal Der LM}
Recall $( M \overset{ f}{\rightarrow} A)$ is a commutative algebra object  and $H = \mathrm{Hom}_R (M ,M)$. Let the $R$-module maps $\pi_1 \colon  \mathrm{Der}_R (A ,M)  \rightarrow \mathrm{Der}_R(A)$, $\pi_2 \colon  \mathrm{Der}_R (A ,M) \rightarrow H  $  be given by   $\pi_1 ( \partial)  := g \circ \partial$ and $\pi_2 ( \partial)  := \partial \circ g$ respectively.  The  universal  \textbf{Lie algebra of derivations} of the   algebra object $(M \overset{ g}{\rightarrow} A)$  is given by
\begin{gather*}
\mathrm{Der}_{\mathcal{LM}} \big( ( M \overset{g}{\rightarrow} A) \big) = \big( \mathrm{Der}_R (A ,M) \xlongrightarrow{\pi_1 + \pi_2}  \mathrm{Der}_R(A) \oplus H  \big),
 \end{gather*}
where  the Lie bracket on $\mathrm{Der}_R(A) \oplus H $ is
\begin{gather*}
[ ( \alpha ,  \beta ) , ( \alpha' , \beta' ) ]_{\mathrm{Der}_R(A) \oplus H }:= \left( [\alpha , \alpha']_{\mathrm{Der}_R(A)} , - [\beta , \beta ' ]_H \right),
\end{gather*}
   and the right  $\mathrm{Der}_R(A) \oplus H $-module structure on $\mathrm{Der}_R (A ,M)$ is given by
\begin{gather}
\label{actionDer}
\partial \otimes ( \alpha , \beta) \longmapsto[  \partial ,  ( \alpha , \beta ) ] := \partial \circ \alpha - \beta \circ \partial .
\end{gather}
\end{Proposition}

\begin{proof}
We f\/irst check that the action  $[ - , - ] $  in \eqref{actionDer} endows $\mathrm{Der}_R (A ,M)$ with a right Lie algebra module structure over (the Lie algebra) $\mathrm{Der}_R(A) \oplus H $.

On the one hand we have
\begin{gather*}
\left[ \partial ,  [ ( \alpha , \beta ), ( \alpha' , \beta' )]_{\mathrm{Der}_R(A) \oplus H } \right] = \left[  \partial ,  \left( [ \alpha , \alpha' ]_{\mathrm{Der}_R(A)},  [ \beta , \beta' ]_H \right)  \right]\\
  \hphantom{\left[ \partial ,  [ ( \alpha , \beta ), ( \alpha' , \beta' )]_{\mathrm{Der}_R(A) \oplus H } \right] }{}
  = \partial \circ  [ \alpha , \alpha' ]_{\mathrm{Der}_R(A)} - [ \beta , \beta' ]_H \circ \partial.
\end{gather*}
On the other hand we have
\begin{gather*}
\left[ \left[ \partial , ( \alpha ,\beta ) \right] ,  ( \alpha' , \beta') \right]   = [ \left( \partial \circ \alpha - \beta \circ   \partial  \right) ,  ( \alpha' , \beta') ] \\
\hphantom{\left[ \left[ \partial , ( \alpha ,\beta ) \right] ,  ( \alpha' , \beta') \right] }{}
= ( \partial \circ \alpha - \beta \circ \partial ) \circ \alpha' - \beta' \circ ( \partial \circ \alpha - \beta \circ \partial ) \\
\hphantom{\left[ \left[ \partial , ( \alpha ,\beta ) \right] ,  ( \alpha' , \beta') \right] }{}
  = \partial \circ \alpha \circ \alpha' - \beta \circ \partial \circ \alpha' - \beta' \circ \partial \circ \alpha + \beta' \circ \beta \circ \partial,
\end{gather*}
so that
\begin{gather*}
\left[  \left[ \partial , ( \alpha ,\beta ) \right]  ,  ( \alpha' , \beta') \right]  -
\left[ \left[ \partial ,  ( \alpha' ,\beta' ) \right] , ( \alpha , \beta ) \right]
 \\
 \qquad {} = \partial \circ \alpha \circ \alpha' - \beta \circ \partial \circ \alpha' - \beta' \circ \partial \circ \alpha
  + \beta' \circ \beta \circ \partial  - \partial \circ \alpha' \circ \alpha  \\
  \qquad\quad{}
  +  \beta' \circ \partial \circ \alpha +  \beta \circ \partial  \circ \alpha' - \beta \circ \beta'  \circ \partial \\
\qquad{} = \partial \circ [ \alpha , \alpha' ]_{\mathrm{Der}_R(A)}  - [ \beta , \beta' ]_H \circ \partial
  = \left[ \partial , \left( [ \alpha , \alpha' ]_{\mathrm{Der}_R(A)} ,  [ \beta , \beta' ]_H \right) \right] .
\end{gather*}
 This shows that the right action of $\mathrm{Der}_R(A) \oplus H$ on $\mathrm{Der}_R (A ,M)$ is well def\/ined.  Hence, the morphism  $ \upvarrho\colon  ( N \overset{ f}{\rightarrow} L) \rightarrow \mathrm{Der}_{\mathcal{LM}} \big( (M \overset{ g}{\rightarrow} A)  \big)$  is  given by the following diagram, which follows from~\eqref{diag hom}
\begin{gather*}
  \xymatrixcolsep{5pc}
  \xymatrixrowsep{2.5pc}
 \xymatrix{
    N  \ar[r]^-{\rho_1}    \ar[d]_{ f}  & \mathrm{Der}_R (A ,M) \ar[d]^{\pi_1 + \pi_2} \\
         L  \ar[r]^-{\rho_0 + \rho_2 }  & \mathrm{Der}_R(A) \oplus H    }
\end{gather*}
which commutes.
\end{proof}

\begin{Remark}
The morphism $\uprho \colon  ( N \overset{ f}{\rightarrow} L) \rightarrow \mathrm{Der}_{\mathcal{LM}} \big( (M \overset{ g}{\rightarrow} A) \big)$ given by   $ (\rho_1 , \rho_0 + \rho_2 )$ is a~morphism of Lie algebras in $\mathcal{LM}$.
\end{Remark}

\begin{Example}
The universal Lie algebra of derivations of the commutative algebra object $(A \overset{\mathrm{id}}{\rightarrow} A)$ is
\begin{gather*}
\big( \mathrm{Der}_R(A) \overset{ }{\longrightarrow} \mathrm{Der}_R(A) \oplus  \mathrm{Der}_R(A) \big).
\end{gather*}
Then, the action of a Lie algebra object $( N \overset{ f}{\rightarrow} L)$ by derivations on  $(A \overset{\mathrm{id}}{\rightarrow} A)$ is given by
\begin{itemize}\itemsep=0pt
\item a Lie algebra map $ \rho_0 \equiv \rho_2 \colon  L  \rightarrow \mathrm{Der}_R(A)$,
\item an $A$-module map $ \rho_1 \colon  N  \rightarrow \mathrm{Der}_R(A)  $
\end{itemize}
satisfying $\rho_1(n) = \rho_0 ( f(n)) $.
\end{Example}

\subsection[Lie-Rinehart algebra objects in $\mathcal{LM}$]{Lie--Rinehart algebra objects in $\boldsymbol{\mathcal{LM}}$}
\label{LR in LM}

A Lie--Rinehart algebra \cite{HuebschmannTerm, RinehartForms} is an algebraic structure which encompasses a~Lie algebra and a~commutative algebra which act on each other in a way that both actions are compatible. This object can be described in any symmetric monoidal category.

In this Section we focus on the description of  Lie--Rinehart algebra objects \textit{in the cate\-go\-ry~$\mathcal{LM}$ of linear maps}.  Based on  \cite[Lemma~3.6]{Loday98} we give a proof of Theorem~\ref{thm1}.

\begin{proof}[Proof of Theorem \ref{thm1}]
Assume the pair $\big( (M \overset{ g}{\rightarrow} A), ( N \overset{ f}{\rightarrow} L) \big)$ is a Lie--Rinehart algebra object. Then  there exist

\begin{itemize}\itemsep=0pt
\item a left $( M \overset{g}{\rightarrow} A)$-module structure on the Lie algebra object $( N \overset{ f}{\rightarrow} L)$,

\item an action $\uprho$ of the Lie algebra object  $( N \overset{ f}{\rightarrow} L)$ on the commutative algebra object $(M \overset{ g}{\rightarrow} A)$ by derivations,

\end{itemize}
and a compatibility condition between these two actions. We now describe what these structures involve.

Firstly, by Proposition \ref{left module}, we deduce that  a left $(M \overset{ g}{\rightarrow} A)$-module structure on the Lie algebra object $( N \overset{ f}{\rightarrow} L)$  turns the $L$-equivariant map $ f\colon N \rightarrow L$ into an $A$-module map. Also, it yields  an $A$-module map $\lambda := \alpha_\ell^N \colon  M \otimes_A L \rightarrow N$. Moreover the Leibniz  algebra structure on~$N$ given by  $[ - , - ]_N$
must    satisfy   $[ n_1 , a \cdot [ n_2 , n_2 ]_N ]_N = 0$ for all $a \in A$ and  $n_1 , n_2 \in N$.

Secondly, by Proposition~\ref{der action}, an  action~$\uprho$ of the Lie algebra object  $( N \overset{ f}{\rightarrow} L)$ on the commutative algebra $(M \overset{ g}{\rightarrow} A)$ by derivations yields
\begin{itemize}\itemsep=0pt
\item an $A$-linear Lie algebra map $ \rho_0 \colon L \rightarrow \mathrm{Der}_R(A)$,
\item an $A$-module  map   $\rho_1\colon N  \rightarrow \mathrm{Der}_R (A ,M) $,
\item an $A$-linear Lie algebra map $  \rho_2\colon L \rightarrow \mathrm{Hom}_R(M,M)$
\end{itemize}
satisfying conditions \eqref{compDer1}, \eqref{compDer3}, and turns  $g\colon M \rightarrow A$ into an $L$-equivariant map. Lastly,   the following  compatibility conditions between the two module structures are  satisf\/ied:{\samepage
\begin{itemize}\itemsep=0pt
\item the pair $(A,L)$ is a  Lie--Rinehart algebra with anchor $\rho_0$,
\item the  right $L$-action on the $A$-module $N$  satisf\/ies
\begin{gather*}
[ a  \cdot  n , b \cdot \xi ] = a \cdot [ n , b \cdot \xi ] - b \cdot \rho_0 ( \xi ) ( a ) \cdot n,
\end{gather*}
which provides a compatibility relation between the right $L$-action on $N$, given by $[ - , - ]$ and the $A$-module structure on $N$.
\end{itemize}
Note that} $\rho_2$ endows $M$ with is a left $(A,L)$-module structure, see~\cite{HuebschmannPaper1} for further details.
\end{proof}

\begin{Example}
Let $(A ,L)$ be a Lie--Rinehart algebra with anchor $\rho_L \colon  L \rightarrow \mathrm{Der}_R(A)$. The pair $\big( (A \overset{\mathrm{id}}{\rightarrow} A) , ( L \overset{\mathrm{id}}{\rightarrow}  L) \big)$ is a Lie--Rinehart algebra object in $\mathcal{LM}$ with $\rho_0 \equiv \rho_1 \equiv \rho_2 = \rho_L $ and $\lambda \equiv \mathrm{id}$.
\end{Example}

\begin{Example}
The pair $ \big( (M \overset{ g}{\rightarrow} A) , \big( \mathrm{Der}_R (A ,M)  \xlongrightarrow{\pi_1 + \pi_2} \mathrm{Der}_R(A) \oplus H \big)  \big) $ is a Lie--Rinehart algebra  object in $\mathcal{LM}$ with
\begin{gather*}
\rho_0  \colon  \ \mathrm{Der}_R(A) \oplus H \rightarrow \mathrm{Der}_R(A), \qquad  \rho_1  \colon  \ \mathrm{Der}_R (A ,M)  \rightarrow \mathrm{Der}_R (A ,M), \\  \rho_2  \colon  \ \mathrm{Der}_R(A) \oplus H \rightarrow H  .
\end{gather*}
\end{Example}

\section[Leibniz algebroids and Lie-Rinehart algebras in $\mathcal{LM}$]{Leibniz algebroids and Lie--Rinehart algebras in $\boldsymbol{\mathcal{LM}}$}

In this short section we prove Theorem \ref{thm2}. This result provides a functorial relation from  Lie--Rinehart algebras to Leibniz algebroids.

\begin{proof}[Proof of Theorem \ref{thm2}]
First note that $N$ is both a right $L$-module and a left $A$-module. Also, recall from \cite{Loday98} that the right $L$-module $N$ becomes a Leibniz algebra by def\/ining the bracket  $[ n_1 , n_2 ]_N = [ n_1 , f (n_2) ]$. Furthermore, since $M$ is a left $(A,L)$-module, by Proposition~\ref{hemi2} we can endow the $A$-module $M \oplus N$ with a Leibniz algebra structure given by the bracket \eqref{eqNbracket}.
Since $f \colon  N \rightarrow L$ and $\rho_0 \colon L \rightarrow \mathrm{Der}_R (A)$ are $A$-linear maps, the map $\rho_{M \oplus N} = -  \rho_0 \circ f$ is also $A$-linear. We now prove that $\rho_{M \oplus N} $ is a  Leibniz algebra antihomomorphism
\begin{gather*}
\rho_{M \oplus N} ([m_1 + n_1 , m_2 + n_2]_{M \oplus N})   = \rho_{M \oplus N} ( - \rho_2( f(n_2)) (m_1) + [ n_1 , n_2 ]_N)  \\
\qquad {} = - \rho_0 (f ([ n_1 , f(n_2) ] ))
  = - \rho_0 ( [f(n_1) , f(n_2) ]_L)  = [ \rho_0 (f (n_2) ) , \rho_0 (f(n_1)) ]_{\mathrm{Der}_R(A)} \\
\qquad{} = [ \rho_{M \oplus N}  (n_2) , \rho_{M \oplus N} (n_1) ]_{\mathrm{Der}_R(A)},
\end{gather*}
so the relations in \eqref{equiv rhoN} hold. We now prove that the Leibniz rule in \eqref{LBrule} hold
\begin{gather*}
[ a \cdot ( m_1 + n_1 ) , m_2 + n_2 ]_{M \oplus N}   = - \rho_2 ( f (n_2 ) ) ( a \cdot m_1 ) + [ a \cdot n_1 , f( n_2) ] \\
 \qquad{}  = - a \cdot \rho_2 ( f( n_2 ) ) \cdot m_1 - \rho_0 ( f (n_2 ) ) ( a ) \cdot m_1
 + a \cdot [ n_1 , f( n_2) ] - \rho_0  ( f(n_2)) (a) \cdot n_1 \\
\qquad{}   = a \cdot (- \rho_2 ( f (n_2 ) ) \cdot m_1 + [ n_1 , f( n_2) ] ) + \rho_{M \oplus N} ( a ) \cdot (m_1 + n_1 ) \\
\qquad{}   = a \cdot [ m_1 + n_1 , m_2 + n_2 ]_{M \oplus N}
+ \rho_{M \oplus N} (m_2 + n_2)( a ) \cdot (m_1 + n_1 ). \tag*{\qed}
  \end{gather*}
  \renewcommand{\qed}{}
\end{proof}

\begin{Example}
Given the Lie--Rinehart algebra
\begin{gather*}
 \big( (M \overset{ g}{\rightarrow} A) , \big( \mathrm{Der}_R (A ,M)  \xlongrightarrow{\pi_1 + \pi_2} \mathrm{Der}_R(A) \oplus H \big)  \big)
\end{gather*}
 in $\mathcal{LM}$, the pairs $(A, \mathrm{Der}_R (A ,M))$ and $( A , M \oplus \mathrm{Der}_R (A ,M) )$
 are Leibniz algebroids.
\end{Example}

\begin{Example}
Given a Lie--Rinehart algebra $(A,L)$, the pair
\begin{gather*}
\big( \big(A \overset{\mathrm{id}}{\rightarrow} A\big), \big(  L \overset{\mathrm{id}}{\rightarrow} L\big) \big)
\end{gather*}
is a Lie--Rinehart algebra object in $\mathcal{LM}$ with corresponding Leibniz algebroid given by $(A, A \oplus L)$ with structure presented in Proposition~\ref{AAL algebroid}.
\end{Example}

\subsection*{Acknowledgements}

The author wishes to thank Uli Kr\"ahmer  for helpful  comments and suggestions. It is also a~pleasure to thank Stuart White for his time,   Jos\'e Figueroa O'Farrill for conversations, and the anonymous referees for their comments and suggestions. This research was funded by an EPSRC DTA grant.

\pdfbookmark[1]{References}{ref}
\LastPageEnding

\end{document}